\documentclass{amsproc}

\newtheorem{theorem}{Theorem}[section]
\newtheorem{lemma}[theorem]{Lemma}
\newtheorem{col}{Corollary}[section]
\theoremstyle{definition}
\newtheorem{definition}[theorem]{Definition}

\newtheorem{assump}{Assumption}
\theoremstyle{remark}
\newtheorem{remark}[theorem]{Remark}

\numberwithin{equation}{section}

\newcommand\bPW{{ {\bf PW}}}

\newcommand\PWoL{{PW_{\negthinspace \omega}(\sqrt{L})}}

\def\PWoL{{{\bf PW}_{\negthinspace \omega}(\sqrt{L})}}
\newcommand\FSL{{F(\sqrt{L})}}
\newcommand\SLB{{(\sqrt{L})}}

\def\EB{{\bf E}}
\def\FB{{\bf F}}
\def\AB{{\bf A}}
\def\HB{{\bf H}}

\begin{document}

\title[Sobolev, Besov and Paley-Wiener vectors ]{Sobolev, Besov and Paley-Wiener vectors in Banach and Hilbert spaces}

\author{Isaac Z. Pesenson}
\address{Department of Mathematics, Temple University,
 Philadelphia,
PA 19122}

\email{pesenson@temple.edu}
\thanks{ I would like to thank Dr. Meyer
 Pesenson for a number of  useful suggestions.}

\dedicatory{Dedicated to 100th Birthday of my teacher  S.G. Krein.}


\begin{abstract}

We consider  Banach spaces equipped with a set of strongly continuous bounded semigroups satisfying certain conditions. Using these semigroups we introduce an analog of  a modulus of continuity and define analogs of Besov norms.  A generalization of a classical interpolation theorem is proven in which the  role of Sobolev spaces is played by  subspaces defined in terms of infinitesimal operators of these semigroups.  We show that our assumptions about a given set of semigroups are satisfied in the case of a strongly continuous bounded representation of a Lie group. In the case of a unitary representation in a Hilbert space we consider an analog of the Laplace operator and use it to define Paley-Wiener vectors. It allows us  to develop a generalization  of the Shannon-type sampling in Paley-Wiener subspaces and to construct Paley-Wiener nearly Parseval frames in the entire Hilbert space. It is shown  that Besov spaces defined previously in terms of the modulus of continuity can be described in terms of approximation by Paley-Wiener vectors and also in terms of the frame coefficients.  Throughout the paper  we extensively use  theory of interpolation and approximation spaces.  The paper ends with applications of our results  to function spaces on homogeneous manifolds. 
\end{abstract}

\maketitle

\section{Introduction and Main Results}
I am honored to have had Selim Grigorievich Krein as my academic advisor and co-author.
Without his steady, strong interest in my research and his support I would not have been able to have the
privilege of becoming a mathematician. His inspiring encouragement  shaped not only my professional trajectory but my life in general. 

\bigskip

The first five sections of this paper are devoted to Sobolev and Besov subspaces  and to relevant moduli of continuity in Banach spaces. This theory is  rooted in my results obtained in 70s:  \cite {Pes78}-\cite{Pes2}, \cite{Pes4}, \cite{KP}. It is a far going generalization of the one-dimensional  theory by J. Lions \cite{Li} and J. Lions-J. Peetre \cite{L-P} which I learned from the  nice book by P. Butzer and H. Berens \cite{BB}. The Paley-Wiener vectors and corresponding approximation theory in abstract Hilbert spaces were introduced in 80s in my papers \cite{Pes3}, \cite{Pes5}, \cite{Pes6}, \cite{KP} (see sections \ref{Hilbert} and \ref{Bes} below). In papers \cite{Pes00}-\cite{Pes11} I used the notion of Paley-Wiener vectors  to  prove Shannon-type sampling theorems on manifolds and in general Hilbert spaces. The construction of frames and description of  Besov subspaces on manifolds and in Hilbert spaces in terms of  frame coefficients is rather recent development and can be found  in   \cite{gp}, \cite{Pes15}. Some of these results were summarized in  \cite{FFP}.

\bigskip

We consider a Banach space $\mathbf{E}$ and operators $D_{1}, D_{2},...,D_{d}$ which generate strongly continuous uniformly bounded semigroups $T_{1}(t), T_{2}(t),...,T_{d}(t), \>\> \|T(t)\|\leq 1, \>\>t\geq 0.$  An analog of a Sobolev space is introduced as the space $\mathbf{E}^{r}$ of vectors in $\mathbf{E} $ for which the following norm is finite
$$
|||f|||_{\mathbf{E}^{r}}=\|f\|_{\mathbf{E}}+\sum_{k=1}^{r}\sum_{1\leq j_{1},  ...j_{k}\leq d}\|D_{j_{1}}...D_{j_{k}}f\|_{\mathbf{E}},
$$
 where $r\in \mathbb{N},\>\> f\in \mathbf{E}.$
  By using the closed graph theorem and
the fact that each $D_{i}$ is a closed operator in $\mathbf{E}$,\ one can show that this norm
 is equivalent to the norm
\begin{equation}\label{Sob}
\|f\|_{r}=\|f\|_{\mathbf{E}}+\sum_{1\leq i_{1},..., i_{r}\leq
d}\|D_{i_{1}}...D_{i_{k}}f\|_{\mathbf{E}},\>\> ~ r\in \mathbb{N}.
\end{equation}
 The mixed modulus of continuity is introduced as
\begin{equation}
\Omega^{r}( s, f)= 
$$
$$
\sum_{1\leq j_{1},...,j_{r}\leq
d}\sup_{0\leq\tau_{j_{1}}\leq s}...\sup_{0\leq\tau_{j_{r}}\leq
s}\|
\left(T_{j_{1}}(\tau_{j_{1}})-I\right)...\left(T_{j_{r}}(\tau_{j_{r}})-I\right)f\|_{\mathbf{E}},\label{M}
\end{equation}
where $f\in \mathbf{E},\ r\in \mathbb{N},  $ and $I$ is the
identity operator in $\mathbf{E}.$   Let $\mathcal{D}(D_{i})$ be the domain of the operator $D_{i}$.
For every $f\in \mathbf{E}$ we introduce a vector-valued function
$$
Tf: \mathbb{R}^{d}\longmapsto \mathbf{E}
$$
defined as
$$
Tf(t_{1}, t_{2},  ..., t_{d})=T_{1}(t_{1}) T_{2}(t_{2})  ... T_{d}(t_{d})f.
$$
\begin{assump}
We assume that the following properties hold.
\begin{enumerate}

\item There exists a set $\mathcal{G}\subset \mathbf{H}^{1}= \bigcap_{i=1}^{d}\mathcal{D}(D_{i})$ which  is dense in $\mathbf{H}$ and  invariant with respect to all $T_{i}(t), \>\>1\leq i\leq d,\>\>t\geq 0.$ 

\item For every $1\leq i\leq d,$ every $f\in \mathcal{G}$ and all $\mathbf{t}=(t_{1}, ..., t_{d})$ in the standard open unit ball $U$ in $\mathbb{R}^{d}$

\begin{equation}\label{assumpt-1}
D_{i}Tf(t_{1}, ..., t_{d})=\sum_{k=1}^{d}\zeta^{k}_{i}(\mathbf{t})\left(\partial_{k} Tf\right)(t_{1}, ..., t_{d}),
\end{equation}
\textit{where $\zeta^{k}_{i}(\mathbf{t})$ belong to $C^{\infty}(U),\>\partial_{k}=\frac{\partial}{\partial t_{k}}.$}

\end{enumerate}

\end{assump}

\begin{remark}

 Since $\mathbf{E}^{1}$ is invariant with respect to all bounded operators $T_{i}(t_{i}),\>\>1\leq i\leq d,$ we obtain for every  $f\in \mathbf{E}^{1}$ 
 \begin{equation}\label{deriv}
 T_{1}(t_{1})T_{2}(t_{2}) ... D_{k}T_{k}(t_{k}) ... T_{d}(t_{d}) f=
 $$
 $$
  T_{1}(t_{1})T_{2}(t_{2}) ... T_{k-1}(t_{k-1})\lim_{s\rightarrow 0}\frac{1}{s}\left(T_{k}(t_{k}+s)-T_{k}(t_{k})\right)T_{k+1}(t_{k+1}) ... T_{d}(t_{d}) f=
  $$
  $$
  \lim_{s\rightarrow 0}\frac{1}{s}T_{1}(t_{1})T_{2}(t_{2}) ... T_{k-1}(t_{k-1})\left(T_{k}(t_{k}+s)-T_{k}(t_{k})\right)T_{k+1}(t_{k+1}) ... T_{d}(t_{d}) f=
$$
  $$
  \frac{\partial}{\partial t_{k}}T_{1}(t_{1})T_{2}(t_{2})  ... T_{d}(t_{d}) f=\left( \frac{\partial}{\partial t_{k}}Tf\right)(t_{1}, ..., t_{d}).
 \end{equation}
Thus the formula (\ref{assumpt-1}) can be rewritten as 
\begin{equation}\label{assumpt}
D_{i}T_{1}(t_{1})T_{2}(t_{2}) ... T_{d}(t_{d})f=\sum_{k=1}^{d}\zeta^{k}_{i}(\mathbf{t})T_{1}(t_{1})T_{2}(t_{2}) ... D_{k}T_{k}(t_{k}) ... T_{d}(t_{d}) f, 
\end{equation}
where $\zeta^{k}_{i}(\mathbf{t})$ belong to $C^{\infty}(U), $ $\mathbf{t}=(t_{1}, ..., t_{d})\in U$.

\end{remark}

\begin{remark}

When semigroups commute with each other  $T_{i}(t_{i})T_{j}(t_{j})=T_{j}(t_{j})T_{i}(t_{i}),\>\>1\leq i,j\leq d,\>\>t_{i}, t_{j} \geq 0,$ then $\zeta^{k}_{i}=\delta_{i}^{k}$.

\end{remark}

It is well known that most of remarkable properties of the so-called Besov functional spaces follow from the fact that they are interpolation spaces (see section \ref{Interp} below) between two Sobolev spaces \cite{BB}, \cite{KPS}.  For this reason we define Besov spaces by the formula
\begin{equation}\label{interp}
 \mathbf{E}^{\alpha, q}=\left(\mathbf{E},\>\mathbf{E}^{r}\right)^{K}_{\alpha/r, \>q},             \ 0<\alpha<r\in \mathbb{N},\ 1\leq p, \>q\leq \infty,
\end{equation}
where $K$ is the so-called Peetre's interpolation functor (see section \ref{Interp} below).
The main result is the following.

\begin{theorem}\label{Main} If Assumption 1 is satisfied  then the following holds true.

\begin{enumerate}

\item The functionals $\Omega_{r}(s, f)$ and $K(s^{r}, f, \mathbf{E}, \mathbf{E}^{r})$  are equivalent. Namely, there exist constants $c>0,\>C>0$, such that for all $f\in \mathbf{E},\>\>s\geq 0$ 
\begin{equation}\label{main-ineq}
c \>\Omega^{r}(s, f)\leq K(s^{r}, f, \mathbf{E}, \mathbf{E}^{r})\leq C \left( \Omega^{r}(s, f)+\min(s^{r}, 1)\|f\| \right).
\end{equation}
\item 
 The norm of the Besov space $\mathbf{E}^{\alpha, q}=\left(\mathbf{E},\>\mathbf{E}^{r}\right)^{K}_{\alpha/r,q},             \ 0<\alpha<r\in \mathbb{N},\ 1\leq p, q\leq \infty,$ is equivalent to the norm 
\begin{equation}\label{Bnorm1}
\|f\|_{\mathbf{E}}+\left(\int_{0}^{\infty}(s^{-\alpha}\Omega^{r}(s,
f))^{q} \frac{ds}{s}\right)^{1/q} , 1\leq q<\infty,
\end{equation}
with the usual modifications for $q=\infty$.
\item The following isomorphism holds true
$
\left(\mathbf{E},\>\mathbf{E}^{r}\right)^{K}_{\alpha/r,q}=\left(\mathbf{E}^{k_{1}},\>\mathbf{E}^{k_{2}}\right)^{K}_{(\alpha-k_{1})/(k_{2}-k_{1}),q},    
$
where $0\leq k_{1}<\alpha<k_{2}\leq r\in \mathbb{N},\ 1\leq  q\leq \infty.$

\item  If $\alpha$ is not integer then the norm
(\ref{Bnorm1}) is equivalent to the norm

\begin{equation}\label{Bnorm3}
\|f\|_{\mathbf{E}^{[\alpha]}}+\sum_{1\leq j_{1},...,j_{[\alpha] }\leq d}
\left(\int_{0}^{\infty}\left(s^{[\alpha]-\alpha}\Omega^{1}
(s,D_{j_{1}}...D_{j_{[\alpha]}}f)\right)^{q}\frac{ds}{s}\right)^{1/q}
\end{equation}
where $[\alpha]$ is the integer part of $\alpha$.
\item   If
$\alpha=k\in \mathbb{N}$ is an integer then the norm
(\ref{Bnorm1}) is equivalent to the norm (Zygmund condition)
\begin{equation}\label{Bnorm4}
\|f\|_{\mathbf{E}^{k-1}}+ \sum_{1\leq j_{1}, ... ,j_{k-1}\leq d }
\left(\int_{0}^{\infty}\left(s^{-1}\Omega^{2}(s,
D_{j_{1}}...D_{j_{k-1}}f)\right)
 ^{q}\frac{ds}{s}\right)^{1/q}.
\end{equation}

\end{enumerate}

\end{theorem}

Next we make another assumption. 
\begin{assump} In addition to Assumption 1 we assume that $\mathbf{E}=\mathbf{H}$ is a Hilbert space and  the following properties  hold.
\begin{enumerate}

\item The operator
 $
L=D_{1}^{2}+... +D_{d}^{2}
$
is a non-negative self-adjoint operator in $\mathbf{H}$.

\item The domain  $\mathcal{D}(L^{k/2}),\>\>k\in \mathbb{N}, $ of the non-negative  square root $L^{k/2}$  coincides  with the space $\mathbf{H}^{k}$ and the  norms (\ref{Sob}) and $\|f\|_{\mathbf{H}}+\|L^{k/2}f\|_{\mathbf{H}}$ are equivalent.

\end{enumerate}

\end{assump}

This assumption allows us to introduce notion of Paley-Wiener vectors (bandlimited vectors) (Definition \ref{PWvector}) and to prove an abstract version of the Paley-Wiener Theorem (Theorem \ref{PWproprties}).

To formulate  an analog of the Shannon-type sampling in abstract Paley-Wiener spaces 
(Theorem \ref{PWSs}) we consider the following assumptions.

\begin{assump}

We assume that  there exist   $C, c>0$ and  $m_{0}\geq 0$ such that for any $0<\rho<1$ there exists a        set of functionals  $\mathcal{A}^{(\rho)}=\left\{\mathcal{A}_{ k}^{(\rho)}\right\}_{k\in \mathcal{K}},$  defined on $ \mathbf{H}^{m_{0}}$, for which
\begin{equation}\label{A}
c\sum_{k}\left|\mathcal{A}_{k}^{(\rho)}(f)\right|^{2}  \leq \|f\|^{2}_{\mathbf{H}}\leq
C\left(\sum_{k\in \mathcal{K}}\left|\mathcal{A}_{ k}^{(\rho)}(f)\right|^{2}+\rho^{2m}\|L^{m/2}f\|_{\mathbf{H}}^{2}\right),
\end{equation}
 for all $ f \in  \mathbf{H}^{m}, \>\>m> m_{0}.$
      
\end{assump}

This Abstract Sampling Theorem \ref{PWSs} is used to construct bandlimited frames in $\mathbf{E}$ (Theorem \ref{frameH}).  
By exploring relations between Interpolation and Approximation spaces we   develop approximation theory by Paley-Wiener vectors in Theorems \ref{approx} and \ref{projections}. 
Finally, in Theorem \ref{framecoef} we obtain description of Besov spaces in terms of frame coefficients.

In section \ref{Lie}  we show that our Assumptions 1 are satisfied for strongly continuous representations of Lie groups in Banach spaces and Assumptions 2 are satisfied for unitary representations of Lie groups in Hilbert spaces. Concerning  Assumptions 3 we  note that our results in \cite{Pes00}-\cite{Pes09a} imply that at least when one is considering a so-called regular or quasi-regular representation of a Lie group in a function space on a homogeneous manifold  $M$, some specific sets of Dirac measures (or even more general functionals \cite{Pes04b})  "uniformly" distributed over $M$ can  serve as functionals $\mathcal{A}^{(\rho)}=\left\{\mathcal{A}_{ k}^{(\rho)}\right\}_{k\in \mathcal{K}_{\rho}},$ where $\rho\in \mathbb{R}_{+}$ represents a specific "spacing" of these Dirac measures.

It should be noted that due to importance of the theory of function spaces  there is constant interest  in extending classical constructions and results from Euclidean to non-Euclidean settings. It is impossible to list even the most significant  publications on this subject which appeared during  the last years. Here we mention just a very few   papers which are  relevant to our work  \cite{DX}, \cite{MY}, \cite{NRT}.

\section{Lie groups and their representations}\label{Lie}

\subsection{Lie groups and their representations in Banach spaces}

Lie algebra $\textbf{g}$ of a Lie group $G$ can be identified with the tangent space $T_{e}(G)$ of $G$ at the identity $e\in G$.
Let
$$
\exp (tX): \> T_{e}(G)\rightarrow G,\>\>t\in \mathbb{R}, \>\>X\in T_{e}(G),
$$
 be the exponential geodesic map i.  e.
$\exp (tX)=\gamma (1),$ where $\gamma (t)$ is a  geodesic of a fixed left-invariant metric on $G$ which 
starts at $e$ with the
initial vector $tX\in T_{e}(G)$:  $\>\gamma (0)=e , \>\>\>\frac{d\gamma (0)}{dt}=tX.$ It is known that $\exp$ is an analytic homomorphism of $\mathbb{R}$ onto one parameter subgroup  $\exp tX$ of  $G$: 
$$
\exp \left((s+t)X\right)= \exp(sX) \exp(tX),\>\>s,t\in \mathbb{R}.
$$
Let  $X_{1},...,X_{d}, \>\>d= dim \>G$ form a basis in the Lie algebra of $G$, then one can consider the following coordinate system in a neighborhood of identity $e$
\begin{equation}
\label{1cs}
(t_{1}, ..., t_{d})\mapsto\exp(t_{1}X_{1}+...+t_{d}X_{d}).
\end{equation}
If $Y_{1}=s_{1}X_{1}+...+s_{d}X_{d}$ and $Y_{2}=t_{1}X_{1}+...+t_{d}X_{d}$ then 
$$
\exp Y_{1}\exp Y_{2}=\exp Z,
$$
where $Z$ is given by the Campbell-Hausdorff formula
\begin{equation}
\label{CH}
Z=Y_{1}+Y_{2}+\frac{1}{2}\left[Y_{1}, Y_{2}\right]+
\frac{1}{12}[Y_{1},[Y_{1},Y_{2}]]-\frac{1}{12}[Y_{2},[Y_{1},Y_{2}]]-\frac{1}{24}[Y_{2},[Y_{1},[Y_{1}, Y_{2}]]]+...   \>\>\>.
\end{equation}
It implies that $Z=\zeta_{1}X_{1}+...+\zeta_{d}X_{d}$, where 
\begin{equation}
\label {CH-2}
\zeta_{j}=s_{j}+t_{j}+ O(\epsilon^{2}),\>\>\>|t_{j}|, |s_{j}|\leq \epsilon, \>\>\>1\leq j\leq d.
\end{equation}
 One can also consider another local coordinate system around $e$ which is given by the formula
\begin{equation}
\label{2cs}
(t_{1}, ..., t_{d})\mapsto\varphi\left(\sum_{j=1}^{d}t_{j}X_{j}\right)=\exp (t_{1}X_{1})...\>\>\exp (t_{d}X_{d}).
\end{equation}

Let us remind that a strongly continuous representation of a Lie group $G$ in a Banach space $\mathbf{E}$ is  a homomorphism $g\mapsto T(g),\>\>\>g\in G,\>\>\>T(g)\in GL(\mathbf{E})$, of $G$ into the group $GL(\mathbf{E}) $ of linear bounded invertible operators in $\mathbf{E}$ such that trajectory $T(g)f, \>\>g\in G,\>\>f\in \mathbf{E}, $ is continuous with respect to $g$ for every $f\in \mathbf{E}$. We will consider only uniformly bounded representations. In this case one can introduce a new norm $\|f\|^{'}_{\mathbf{E}}=\sup_{g\in G}\|T(g)f\|_{\mathbf{E}},\>\>f\in \mathbf{E},$ in which $\|T(g)f\|^{'}_{\mathbf{E}}\leq \|f\|^{'}_{\mathbf{E}}$.  Thus, without any restriction  we will assume that the last inequality  is satisfied in the original norm $\|\cdot\|_{\mathbf{E}}$.
Let  $X_{1}, ... , X_{d}$ be a basis in $\mathbf{g}$.  With every $X_{j}, \>\>1\leq j\leq d,$ one associates a  strongly continuous one-parameter group of isometries $t\mapsto T(\exp t X_{j}),\>\>t\in \mathbb{R},$ whose generator is denoted as $D_{j},\>\>1\leq j\leq d$.

\begin{lemma}
If $T: G\mapsto GL(\mathbf{E})$ is a strongly continuous bounded representation of $G$ in a Banach space $\mathbf{E}$ and $T_{j}(t)=T(\exp t X_{j})$, where $\{X_{1}, ..., X_{d}\}$ is a basis in $\mathbf{g}$ then the Assumption 1 is satisfied for groups $T_{j}$ and their infinitesimal operators $D_{j},\>\>1\leq j\leq d$.

\end{lemma}

\begin{proof}
The fact that the Garding space $\mathcal{G}\subset \mathbf{E}^{1}$ is dense in $\mathbf{E}$ and invariant is well known \cite{Nel}, \cite{NS}.
Since $\exp$ and $\varphi$ are diffeomorphisms in a neighborhood of zero in $\textbf{g}$ the map $\exp^{-1}\circ \varphi: \textbf{g}\mapsto \textbf{g}$ is also a diffeomorphism. 
The formulas (\ref{CH}) and (\ref{CH-2})  give connection between (\ref{1cs}) and (\ref{2cs}) 
\begin{equation}\label{1-2cs}
\varphi\left(\sum_{j=1}^{d}t_{j}X_{j}\right)=\exp \left(\sum_{k=1}^{d}\alpha_{k}(t)X_{k}\right),\>\>t=(t_{1},,, t_{d}),
\end{equation}
where 
\begin{equation}
\label{3cs}
\alpha_{j}(t)=t_{j}+O(\epsilon^{2}),\>\>\>|t_{j}|\leq \epsilon,\>\>\>1\leq j\leq d.
\end{equation}
In particular,  (\ref{CH}) implies
$$
\exp \tau X_{j}\exp\sum_{i=1}^{d}t_{i}X_{i}=\exp\sum_{k=1}^{d}\gamma_{k}^{j}(t,\tau)X_{k},\>\>\>t=(t_{1}, ... ,t_{d}),
$$
where
\begin{equation}\label{gamma}
\gamma_{k}^{j}(t,\tau)=t_{k}+\tau\zeta_{k}^{j}(t)+\tau^{2}R_{k}^{j}(t,\tau),
\end{equation}
and $\zeta_{k}^{j}(t)=\delta^{j}_{k}+Q_{k}^{j}(t) $, where $\delta_{k}^{j}$ is the Kronecker symbol and $ Q_{k}^{j}(t) $ and $R_{k}^{j}(t, \tau)$  are  convergent series in $t_{1}, ... ,t_{d}$ and  $t_{1}, ... ,t_{d}, \tau$ respectively.

Since  for $f\in \mathbf{E}^{1}$ one has $D_{j}T(g)f=\frac{d}{d\tau}T\left(\exp \tau X_{j}\right)T(g)f|_{\tau=0}$
we obtain for $f\in \mathbf{E}^{1}$ the following
$$
D_{j}T_{1}(t_{1})...T_{d}(t_{d})f=D_{j}T\left( \varphi \left(\sum_{i=1}^{d}t_{i}X_{i}\right)\right)=D_{j}T\left(  \exp\sum_{i=1}^{d}\alpha_{i}(t)X_{i}\right)f=
$$
$$
\frac{d}{d\tau}T\left(\exp \tau X_{j}\right)T\left( \varphi \left(\sum_{i=1}^{d}t_{i}X_{i}\right)\right)f|_{\tau=0}=\frac{d}{d\tau}T\left(\exp\sum_{i=1}^{d}\gamma_{i}^{j}(\alpha, \tau)X_{i}\right)f|_{\tau=0},
$$
where $\alpha=\left(\alpha_{1}(t), ... , \alpha_{d}(t)\right)$ and according to  (\ref{gamma}) $ \gamma_{i}^{j}(\alpha, \tau)=\alpha_{i}(t)+\tau\zeta_{i}^{j}(\alpha(t))+\tau^{2}R_{i}^{j}(\alpha, \tau)$.
By using the Chain Rule and (\ref{1-2cs})  we finally obtain the formula (\ref{assumpt-1})
$$
D_{j}T_{1}(t_{1})...T_{d}(t_{d})f=\frac{d}{d\tau}T\left(\exp\sum_{i=1}^{d}\gamma_{i}^{j}(\alpha, \tau)X_{i}\right)f|_{\tau=0}=
$$
$$
\sum_{k=1}^{d}\left(\frac{d}{d\tau}\gamma_{k}^{j}(\alpha, \tau)|_{\tau=0}\right)\partial_{k}T\left(\exp\sum_{i=1}^{d}\gamma_{i}^{j}(\alpha, \tau)X_{i}\right)f|_{\tau=0}=\sum_{k=1}^{d}\zeta_{k}^{j}(t)\partial_{k}T_{1}(t_{1})... T_{d}(t_{d})f.
$$
 Lemma is proved.
\end{proof}

\subsection{Unitary representations in Hilbert spaces}

A strongly continuous unitary representation of a Lie group $G$ in a Hilbert space $\mathbf{H} $ is a homomorphism  $T: G \mapsto U(\mathbf{H})$ where $U(\mathbf{E})$ is the group of unitary operators of $\mathbf{H}$ such that $T(g)f,\>\>g\in G,$ is continuous on $G$ for any $f\in \mathbf{H}$.  The Garding space $\mathcal{G}$ is defined as the set of vectors $h$  in $\mathbf{H}$ that have the representation
$
h=\int_{G}\varphi(g)T(g)f dg,
$
where $f\in \mathbf{H}$, $\>\>\varphi\in C_{0}^{\infty}(G)$, $\>\>dg$ is a left-invariant measure on $G$.  If $X\in \mathbf{g}$ is identified with a  right-invariant vector field
$$
X\varphi(g)=\lim_{t\rightarrow 0}\frac{\varphi\left( \exp tX \cdot g\right)-\varphi(g)}{t},
$$ 
then one has a representation $D(X)$ of $\mathbf{g}$ by operators which act  on $\mathcal{G}$ by the formula 
$
D(X)h=-\int_{G}X\varphi(g)T(g)fdg.
$
It is known  that 
$
\mathcal{G}\subset \bigcap_{r\in N}\mathbf{H}^{r}=\mathbf{H}^{\infty}
$
 is invariant with respect  to all operators $D(X), \>\>X\in \textbf{g},$ and dense in every $\mathbf{H}^{r}$.
 If $X_{1}, ..., X_{d}$ is a basis in $\mathbf{g}$ and $D_{i}=D(X_{i}),\>\>1\leq i \leq d,$ we consider the operator $L_{\mathcal{G}}=-\sum_{i=1}^{d}D_{i}^{2}$ defined on $\mathcal{G}$.

 Since $L_{\mathcal{G}}$ is symmetric and the differential operator $-\sum_{i=1}^{d}X_{i}^{2}$ is elliptic  on the group $G$ the Theorem 2.2 in \cite{NS} implies that $L_{\mathcal{G}}$ is essentially self-adjoint, which means $\overline{L}_{\mathcal{G}}=L_{\mathcal{G}}^{*}$. In other words, the closure $\overline{L}_{\mathcal{G}}=L$ of $L_{\mathcal{G}}$ from $\mathcal{G}$ is a self-adjoint operator. Obviously, $L\geq 0$. We introduce the self-adjoint operator 
 $
 \Lambda=I+L\geq 0.
 $

\begin{theorem}
The space $\mathbf{H}^{r}$ with the norm (\ref{Sob}) is isomorphic  to the domain of $\Lambda^{r/2}$ with the norm $\|\Lambda^{r/2}f\|_{\mathbf{H}}.$
\end{theorem}

\begin{proof}
In the case $r=2k,$ the inequality 
 \begin{equation}\label{2.9}
  \|f\|_{\mathbf{H}^{2k}}\leq C(k)\|\Lambda^{k}f\|_{\mathbf{H}}
\end{equation}
is shown in \cite{Nel}, Lemma 6.3. The reverse inequality is obvious.   
We consider now the case $r=2k+1$.  If $f\in \mathbf{H}^{2}=\mathcal{D}(\Lambda)$, then since 
$
\mathcal{D}(\Lambda)\subset \mathcal{D}(\Lambda^{1/2})
$
we have
\begin{equation}
\label{1/2}
\|f\|_{\mathbf{H}}
^{2}+\sum_{j}\|D_{j}f\|_{\mathbf{H}}
^{2}=\left<f,f\right>+\sum_{j}\left<D_{j}f,D_{j}f\right>=\left<f,f\right>+\left<-\sum_{j}D_{j}^{2}f,f\right>=
$$
$$
\left<f-\sum_{j}D_{j}^{2}f,f\right>=\left<\Lambda f,f\right>=\|\Lambda^{1/2}f\|_{\mathbf{H}}
^{2}.
\end{equation}
These equalities imply that $\mathbf{H}^{1}$ is isomorphic to $\mathcal{D}(\Lambda^{1/2})$.
Our goal is to to prove existence of an isomorphism between $\mathbf{H}^{2k+1}$ and $\mathcal{D}(\Lambda^{k+1/2})$. It is enough to establish equivalence of the corresponding norms on the set $\mathbf{H}^{4k+2}=\mathcal{D}(\Lambda^{2k+1})$ since the latest is dense in $\mathbf{H}^{2k+1}$.
If $f\in \mathbf{H}^{4k+2}\subset \mathbf{H}^{2k}$ then $D_{j}f\in \mathbf{H}^{4k+1}\subset \mathbf{H}^{2k}$ and 
$
\Lambda^{k}f=\sum_{m\leq k}\sum D_{j_{1}}^{2}...D_{j_{m}}^{2}f.
$
Thus if $f\in \mathbf{H}^{4k+2}$ then
$$
\left\|D_{j_{1}}...D_{j_{2k+1}}f  \right \|_{\mathbf{H}}
\leq C\left\|\Lambda^{k}D_{j_{2k+1}}f \right\|_{\mathbf{H}}
=\left\|\sum_{m\leq k}\sum D_{j_{1}}^{2}...D_{j_{m}}^{2}D_{j_{2k+1}}f \right\|_{\mathbf{H}}
.
$$
Multiple applications of the identity $D_{i}D_{j}-D_{j}D_{i}=\sum_{k} c_{i,j}^{k}D_{k}$ which holds on $\mathbf{H}^{2}$  lead to the inequality
$
\left\|D_{j_{1}}...D_{j_{2k+1}}f \right \|_{\mathbf{H}}
 \leq C\left(\|D_{j_{2k+1}}\Lambda^{k}f\|_{\mathbf{H}}
+\|Rf\|_{\mathbf{H}}
\right),
$
where $R$ is a polynomial in $D_{1},...,D_{d}$ whose degree $\leq 2k$. According to (\ref{2.9}) and (\ref{1/2}) we have that 
$$
\left\|D_{j_{2k+1}}\Lambda^{k}f \right\|_{\mathbf{H}}
 \leq  \left\|\Lambda^{1/2}\Lambda^{k}f \right\|_{\mathbf{H}}
=\left\|\Lambda^{k+1/2}f \right\|_{\mathbf{H}}
$$
and also
$
\left\|Rf \right \|_{\mathbf{H}}
\leq \|f\|_{\mathbf{H}^{2k}}\leq C(k)\left\|\Lambda^{k}f \right\|_{\mathbf{H}}
.
$
Since $\|\Lambda^{k}f\|_{\mathbf{H}}
$ is not decreasing with $k$ we get the following estimate 
$$
\|D_{j_{1}}...D_{j_{2k+1}}f\|_{\mathbf{H}}
\leq C(k)\|\Lambda^{k+1/2}f\|_{\mathbf{H}}
, \>\>\>f\in \mathbf{H}^{4k+2}.
$$
Now, since for $f\in \mathbf{H}^{4k+2}$ we have 
$
D_{j_{1}}... D_{j_{2k}}f\in \mathbf{H}^{2k+2}\subset \mathbf{H}^{1}=\mathcal{D}(\Lambda^{1/2}),
$
and the equality 
$
\Lambda^{k}f=\sum_{m\leq k}\sum D_{j_{1}}^{2}...D_{j_{m}}^{2}f,
$
holds we obtain, by using (\ref{1/2}) 
$$
\|\Lambda^{k+1/2}f\|_{\mathbf{H}}
=\|\Lambda^{1/2}\sum_{m\leq k}\sum D_{j_{1}}^{2}...D_{j_{m}}^{2}f\|_{\mathbf{H}}
\leq C\|f\|_{\mathbf{H}^{2k+1}},\>\>\>\>C=C(k).
$$
Theorem is proved.  
\end{proof}

\begin{col}
If $T$ is a strongly continuous unitary representation of a Lie group in a Hilbert space $\mathbf{H}$ and $X_{1},  ..., X_{d}$ is a basis in the corresponding algebra Lie $\mathbf{g}$ then for $T_{j}(t)=T(\exp tX_{j}),\>\>1\leq j\leq d,$ and their generators $D_{j}, \>\>1\leq j\leq d,$ the Assumption 2 is satisfied.
\end{col}

 \section{Hardy-Steklov operator associated with operators $D_{1}, D_{2}, ..., D_{d}$.}

We return to the general set up. Namely, we consider a Banach space $\mathbf{E}$ and operators $D_{1}, D_{2},...,D_{d}$ which generate strongly continuous uniformly bounded semigroups $T_{1}(t), T_{2}(t),...,T_{d}(t), \>\> \|T(t)\|\leq 1, \>\>t\geq 0.$ It is assumed that the Assumption 1 holds.

 \subsection{Preliminaries} 
 
 If $D$ generates in $\mathbf{E}$ a strongly continuous bounded semigroup $T_{D}(t),\>\>t\geq 0, $ and 
$$
\Omega^{r}( s, f)= 
\sup_{0\leq\tau\leq s}\|\left(T_{D}(\tau)-I\right)^{r}f\|_{\mathbf{E}},
$$
where $I$ is the identity operator, then one can easily prove  the following two inequalities 
\begin{equation}
\Omega^{m}\left(f, s\right)\leq s^{k}\Omega^{m-k}(D^{k}f,
s), \label{???}
\end{equation}
and 
\begin{equation}
\Omega^{m}\left(f, as\right)\leq \left(1+a\right)^{m}\Omega^{m}(f,
s),\>\> a\in \mathbb{R}.\label{???}
\end{equation}

Let $F(x_{1}, x_{2},...,x_{N})$ be a function on
$\mathbb{R}^{N} $
 that takes values in  the Banach space $\mathbf{E}$
 $$
 F: \mathbb{R}^{N}\longmapsto \mathbf{E}.
 $$
  For $1\leq i\leq N$ we introduce  difference operator
by the
 formula
 \begin{equation}\label{dif-op}
 (\Delta_{i}(s)F)(x_{1}, x_{2},...,x_{N})=F(x_{1},
x_{ 2},...,x_{i-1}, s, x_{i+1}, ..., x_{N})- 
 $$
  $$
  F(x_{1}, x_{2},...,
 x_{i-1}, 0, x_{i+1}, ..., x_{N}).
 \end{equation}
  For a scalar differentiable function $\theta: \mathbb{R}^{N} \longmapsto  \mathbb{R}$  and for $1\leq i_{1},..., i_{l}\leq N,\>\>i_{m}\neq i_{k}$, the following inequality holds
$$
\max_{0\leq x_{j}\leq
s_{j}}|\Delta_{i_{1}}(s_{i_{1}})...\Delta_{i_{l}}(s_{i_{l}})
 \theta (x_{1},...,x_{N})| \leq
s_{i_{1}}...s_{i_{l}}\max_{0\leq x_{j}\leq s_{j}}\left|\frac{\partial^{k}}{\partial_{i_{1}}...\partial_{i_{k}}}
 \theta(x_{1},...,x_{N})\right|.
 $$
 One has for $1\leq i\leq N$
\begin{equation}\label{product}
\Delta_{i}(s)(\theta F)(x_{1},...,x_{N})=
\Delta_{i}(s)\theta(x_{1},...,x_{N})F(x_{1},... ,x_{i-1}, s, x_{i+1}, ..., x_{N})+
$$
$$
\theta (x_{1},... ,x_{i-1}, 0, x_{i+1}, ..., x_{N})\Delta_{i}(s)F(x_{1}, ..., x_{N}),
\end{equation}
and then if $i_{m}\neq i_{k}$  one has
\begin{equation}\label{Long}\Delta_{i_{1}}(s_{i_{1}})...\Delta_{i_{l}}(s_{i_{l}})(\theta
F)(x_{1},..., x_{N})
 =
 $$ 
 $$
 \sum \Delta_{i_{1}^{(1)}}(s_{i_{1}^{(1)}})...\Delta_{i_{l}^{(1)}}(s_{i_{l}^{(1)}})
  \theta(x^{(0)}
 _{i_{1}^{(2)},...,i_{l}^{(2)}})   \Delta_{i_{1}^{(2)}}(s_{i_{1}^{(2)}})...\Delta_{i_{l}^{(2)}}(s_{i_{l}^{(2)}})
F(x^{(s)}
 _{i_{1}^{(1)},...,i_{l}^{(1)}}),
\end{equation}
  where the sum is taken over all possible partitions of the natural vector
 $(i_{1},...,i_{l})$ in the sum of nonnegative integer vectors
$(i^{(1)}_{1},... ,i^{(1)}_{l})$
 and $(i_{1}^{(2)},...,i^{(2)}_{l})$. 
 The
 $x^{(0)}_{i_{1}^{(2)},...,i_{l}^{(2)}}$ denotes a vector obtained from the vector
  $(x_{1},...,x_{N})$ by replacing the
 coordinates 
 $x_{i_{1}^{(2)},...,i_{l}^{(2)}}$ by $0$ and
 $x^{(s)}_{i_{1}^{(1)}
 ,..., i_{l}^{(1)}},$ denotes the vector obtained by replacing
$i_{1}^{(1)},...,i_{l}^{(1)}$ by
$s_{i^{(1)}_{1}},...,s_{i^{(1)}_{l}},$ respectively.

If $F$ is a  differentiable function  and $m\neq i$ then we have
\begin{equation}\label{commute-0}
\Delta_{i}(s)\frac{\partial}{\partial x_{m}}F(x_{1},  x_{2},..., x_{N})=\frac{\partial}{\partial x_{m}}\Delta_{i}(s)F(x_{1}, x_{2},...,x_{N}), \>\>\>m\neq i.
\end{equation}
 In particular, for any $r\geq 2$ for $\tau_{j}=\tau_{j,1}+...+\tau_{j,r}$, $j=1,...,d,\>1\leq i, i^{'}\leq d,\>1\leq k, k^{'}\leq r, $ if $(i^{'}, k^{'})\neq (i,k)$
we have 
 $$
 \frac{\partial}{\partial \tau_{i^{'}, k^{'}}}\Delta_{i,k}(s)\prod_{j=1}^{d}T_{j}\left(\tau_{j}\right)f=
 \Delta_{i,k}(s) \frac{\partial}{\partial \tau_{i^{'}, k^{'}}}\prod_{j=1}^{d}T_{j}\left(\tau_{j}\right)f, \>\>\>f\in \mathbf{E}^{1},
 $$
 or, equvalently,
 $$
 \frac{\partial}{\partial \tau_{i^{'}}}\Delta_{i,k}(s)\prod_{j=1}^{d}T_{j}\left(\tau_{j}\right)f=
 \Delta_{i,k}(s) \frac{\partial}{\partial \tau_{i^{'}}}\prod_{j=1}^{d}T_{j}\left(\tau_{j}\right)f, \>\>\>f\in \mathbf{E}^{1},
 $$  
 where the right-hand side and $
\Delta_{i,k}(s)\prod_{j=1}^{d}T_{j}\left(\tau_{j}\right)f 
 $ are defined  according to (\ref{dif-op}).
 In the same notations $\tau_{j}=\tau_{j,1}+...+\tau_{j,r},\>1\leq m, i\leq d,\>1\leq k\leq r,$ we clearly have the identity
 $$
  D_{m}\Delta_{i,k}(s) \prod_{j=1}^{d}T_{j}\left(\tau_{j}\right)f= \Delta_{i,k}(s)  D_{m}\prod_{j=1}^{d}T_{j}\left(\tau_{j}\right)f,\>\>\>f\in \mathbf{E}^{1},
  $$
 which implies   for every $f\in \mathbf{E}^{1}$ the next formula
 \begin{equation}\label{commute}
 D_{m}\Delta_{i_{1}, k_{1}}(s)..\Delta_{i_{l}, k_{l}}(s)\prod_{j=1}^{d}T_{j}\left(\tau_{j}\right)f=\Delta_{i_{1}, k_{1}}(s)..\Delta_{i_{l}, k_{l}}(s) D_{m}\prod_{j=1}^{d}T_{j}\left(\tau_{j}\right)f.
 \end{equation}
 Using (\ref{commute}),   (\ref{assumpt-1}), (\ref{Long}) and (\ref{commute-0}) 
we obtain for $f\in\mathbf{E}^{1}$,  $\tau_{j}=\tau_{j,1}+...+\tau_{j,r}$, $j=1,...,d,$ $\>r\geq l+1, \>1\leq m\leq d$,
\begin{equation}\label{long}
D_{m}\Delta_{i_{1}, k_{1}}(s)...\Delta_{i_{l}, k_{l}}(s)\prod_{j=1}^{d}T_{j}\left(\tau_{j}\right)f=
$$
 $$\sum \sum _{i_{l+1}=1}^{d}\Delta_{i_{1}^{(1)},k_{1}^{(1)}}(s)...\Delta_{i_{l}^{(1)},
 k_{l}^{(1)}}(s)\zeta^{i_{l+1}}_{m}(...)    \frac{\partial}{\partial\tau_{i_{l+1}}}
 \Delta_{i_{1}^{(2)},k_{1}^{(2)}}(s)...\Delta_{i_{l}^{(2)},k_{l}^{(2)}}(s)\prod_{j=1}^{d}T_{j}\left(...\right)f,
\end{equation}
where outer summation $\sum$ and arguments of $\zeta^{i_{l+1}}_{m}(...)$ and $\prod_{j=1}^{d}T_{j}\left(...\right)f$ the same as in (\ref{Long}).

\subsection{The Hardy-Steklov operator}

We introduce a generalization of the classical Hardy-Steklov operator.
For a positive small $s$, natural $r$ and $1\leq j\leq d$ we set

$$ 
H_{j.r}(s)f=(s/r)^{-r}\int_{0}^{s/r}...\int_{0}^{s/r}\sum
_{k=1}^{r}(-1)^{k}
C^{k}_{r}T_{j}( k(\tau_{j,1}+...+\tau_{j,r})fd\tau_{j,1}...d\tau_{j,r},
$$
where $C^{k}_{r}$ are the binomial coefficients and then define the Hardy-Steklov operator: 
$H_{r}(s)f=\prod_{j=1}^{d}H_{j,r}(s)f=
H_{1,r}(s)H_{2,r}(s)...H_{d,r}(s)f,\>\>\>\>f\in \mathbf{E}.
 $
For every fixed $f\in \mathbf{E}$ the function $H_{r}(s)f$ is an abstract valued function from  $\mathbb{R}$ to $\mathbf{E}$ and it is a linear combination of  some abstract valued functions  
of the form

\begin{equation}
\label{genterm}
(s/r)^{-rd}\int_{0}^{s/r}...\int_{0}^{s/r}Tf(\tau) d\tau_{1,1}...d\tau_{d,r},
\end{equation}
where 
\begin{equation}\label{10}
\tau=(k_{1}\tau_{1}, k_{2}\tau_{2},...,k_{d}\tau_{d}),\>\>\>\>1\leq k_{j}\leq r,
\end{equation}
\begin{equation}\label{20}
\tau_{j}=(\tau_{j,1}+\tau_{j,2}+...+ \tau_{j,r}), \>\>\>\>1\leq j\leq d,
\end{equation}
and 
\begin{equation}
\label{3}
Tf(\tau)=T_{1}(k_{1}\tau_{1})T_{2}( k_{2}\tau_{2}) ... \>T_{d}(k_{d}\tau_{d})f.
\end{equation}

\begin{lemma}\label{Mlemma}The following holds:

\begin{enumerate}

\item For every $f\in \mathbf{E}$ the function $H_{r}(s)f$ maps $\mathbb{R}$ to $\mathbf{E}^{r}$.

\item For every $0\leq q\leq r$ the "mixed derivative" 
$D_{j_{1}} ... D_{j_{q}}H_{r}(s)f , \>\>\>1\leq j_{k}\leq d $ is another abstract valued function with values in $\mathbf{E}^{r-q}$  and it is a linear combination of abstract valued functions  (with values in $\mathbf{E}^{r}$) of the form
\begin{equation}\label{lemma-formula}
(s/r)^{-rd}\underbrace{\int_{0}^{s/r}.....\int_{0}^{s/r}}_{rd-m}\mu_{j_{1},...,j_{q}}^{i_{1},...,i
_{l}; m} ( \cdot)\Delta_{i_{1},k_{1}}(s/r) .... \Delta_{i_{l},k_{l}}(s/r)Tf(\cdot) d\cdot,
\end{equation}
 where
 \begin{equation}\label{estimate}
  \max_{0\leq\tau_{i,j}\leq s}|\partial ^{p} \mu_{j_{1},...,j_{q}}^
 {i_{1},...,i_{l}; m} ( \cdot )|\leq cs^{m-l},\>\>\>p\in \mathbb{N}\cup \{0\},
\end{equation}
and  $0<s<1,\>\>0\leq m\leq rd,\>\>0\leq l\leq m,$  where $l=0$ corresponds to the case when the set of indices $\{i_{1},..., i_{l}\}$ is empty.
\end{enumerate}

  \end{lemma}

  \begin{proof}
  
  The proof proceeds by induction on $q$. We shall   show that
$H_{r}(s)f$
  is in $\mathbf{E}^{1}$ for every $f\in \mathbf{E}$.  Let's assume first that $f\in \mathbf{E}^{1}$.   
  Since every $D_{j},\>\>1\leq j\leq d,$ is a closed operator to show that the term  (\ref{genterm}) belongs
   to $\mathcal{D}(D_{j})$  it  is sufficient  to show existence of the integral  (see notations  (\ref{10}), (\ref{20}))
   \begin{equation}
   \label{preparation-1}
 \left(\frac{s}{r}\right)^{-rd}\int_{0}^{s/r}...\int_{0}^{s/r}D_{j}Tf(\tau)d\tau_{1,1}...d\tau_{d,r}.
   \end{equation}
According to (\ref{assumpt-1})  the last integral equals to 
   \begin{equation}
   \label{preparation}
\sum_{i=1}^{d} \left(\frac{s}{r}\right)^{-rd}\int_{0}^{s/r}...\int_{0}^{s/r}\zeta_{j}^{ i}(\tau)\partial_{i}Tf  (k_{1}\tau_{1}, ... , k_{d}\tau_{d})d\tau_{1,1}...d\tau_{d,r},
\end{equation}
where $\tau$ and $\>\>\tau_{j}$ described in    (\ref{10}) and (\ref{20}).                Since derivative $\partial_{i}$ is the same as the derivative $\frac{\partial}{\partial \tau_{i_{1}, 1}}$
the integration by parts  formula and (\ref{product}) 
  allow to continue (\ref{preparation}) as follows
  \begin{equation}\label{step 1}
  (s/r)^{-rd}\sum_{i=1}^{d}\left(\underbrace{\int_{0}^{s/r}...\int_{0}^{s/r}}_{rd-1}\zeta ^{i}_{j}
  (\tau^{(s/r)}_{i})\Delta_{i}(s/r)Tf(\tau )(d\tau)_{i}\right) +
  $$
  $$
   (s/r)^{-rd}\sum_{i=1}^{d}\left(\underbrace{ \int _{0}^{s/r} ... \int_{0}^{s/r}}_{rd-1}(\Delta_{i}(s/r)\zeta_{j}^{i}(\tau))
  Tf(\tau_{i}^{(0)})(d\tau)_{i}\right) -
  $$
  $$
  (s/r)^{-rd}\sum_{i=1}^{d}\left( \underbrace{\int_{0}^{s/r}
...\int_{0}^{s/r}}_{rd}(\partial_{i}\zeta^{i}_ {j})
  (\tau)Tf(\tau)d\tau \right)= B_{j}(s)f,
\end{equation}
  where $\tau=(\tau_{1,1}+...+\tau_{1,r}, \>... \>,  \tau_{n,1}+ ... +\tau_{n,r}),\>\>\>\>\tau_{i}^{s/r}=(\tau_{1,1}+...+\tau_{1,r}, \>... , \tau_{i,1}+...+\tau_{i,r-1}+s/r, ... , \tau_{n,1}+ ... +\tau_{n,r}),\>\>\>\>\tau_{i}^{0}=(\tau_{1,1}+...+\tau_{1,r}, \>... , \tau_{i,1}+...+\tau_{i,r-1}+0, ... , \tau_{n,1}+ ... +\tau_{n,r}), \>\>\>\>d\tau=d\tau_{1,1}....d\tau_{n,r},$
  and $ (d\tau)_{i}=d\tau_{1,1}.... d\tau_{i, r-1} \widehat{d\tau_{i,r}} d\tau_{i+1,1} ... d\tau_{n,r},$     where the term $d\tau_{i,r}$ is missing.

  Since integrand of each of the three integrals is bounded it implies  existence of (\ref{preparation-1}) which, in turn, shows that (\ref{genterm}) is an element of $\mathcal{D}(D_{j})$ for every $1\leq j\leq d$.

 Thus if $f$ belongs to $\mathbf{E}^{1}$ then 
  $H_{r}(s)f$ takes values in $\mathbf{E}^{1}$ and the formula
  $D_{j}H_{r}(s)f=B_{j}(s)f$ holds where  the operator
$B_{m}(s)$ is bounded.  This fact along with the fact that  $\mathbf{E}^{1}$ is dense in $\mathbf{E}$ implies the formula  $D_{j}H_{r}(s)f=B_{j}(s)f$ for every $f\in \mathbf{E}$. Thus we proved the first part of the Lemma for $q=1$.

To verify the second claim of the Lemma it is sufficient to note that in the first line of (\ref{step 1}) one has $m=l=q=1,  \>\>j=j_{1}$, $\>\>\mu_{j_{1}}^{1,1}=\zeta^{i}_{j}$,  in the second line $m=1,\>l=0,\>\>q=1,\>\>j=j_{1},\>\>\mu_{j_{1}}^{0,1}=\Delta_{i}(s/r)\zeta_{j}^{i}$, in the third line $m=l=0,\>q=1, \>\mu_{j_{1}}^{0,0}= \partial_{i}\zeta^{i}_ {j}$. It is easy to verify that in all of these cases the  estimates (\ref{estimate}) hold.

  We now assume that theorem is proved for all $q<r$ and we shall prove
it for $ q+1\leq r$.
  Suppose $f\in \mathbf{E}^{1}$. We consider

\begin{equation}\label{integral-0}
\int_{0}^{s/r}.....\int_{0}^{s/r}\mu_{j_{1},...,j_{q}}^{i_{1},...,i_{l}; m} (\cdot
) D_{j_{q+1}}     \Delta_{i_{1},k_{1}}(s/r) ....
\Delta_{i_{l},k_{l}}(s/r)T(\cdot)fd\cdot  . 
\end{equation}

By (\ref{long}) this expression is decomposed into a sum of terms of the
form

\begin{equation}\label{integral-1}
\int
_{0}^{s/r}...\int_{0}^{s/r}\nu_{j_{1},...,j_{q+1}}^{i_{1},...,i_{l+1};
m} (\cdot)
\partial_{i_{l+1}}\Delta_{i^{(2)}_{1},k^{(2)}_{1}}(s/r)
...\Delta_{i^{(2)}_{l},k^{(2)}_{l} }(s/r) T(\tau)fd\tau,
\end{equation}
 where $\nu_{j_{1},...,j_{q+1}}^{i_{1},...,i_{l+1};
m}$
are constructed from $\mu_{j_{1},...,j_{q}}^{i_{1},...,i_{l}; m}$ in accordance with (\ref{long}). Integrating by
parts as in the
 first step of  induction, we obtain that the integral (\ref{integral-1}) exists and
is equal to the
 value of some bounded operator at an element $ f\in \mathbf{E}^{1}$. As a
result we obtain that
 $H_{r}(s)$ belongs to the space $\mathbf{E}^{q+1}$ and also that
$D_{j_{1}}...D_{j_{q+1} }H_{r}(s)f$
 is a linear combination of terms of the form (\ref{lemma-formula}) . The rest of the theorem fellows from the
induction hypothesis.
 Theorem is proved.
\end{proof}

\section{Interpolation  spaces}\label{Interp}

\noindent
The goal of the section is to introduce basic notions of the theory of interpolation spaces \cite{BL}, \cite{BB},  \cite{KPS}. Later in section \ref{Bes} we will also introduce the so-called approximation spaces \cite{PS}, \cite{BSch}.  
It is important to realize that the relations between
interpolation and approximation spaces cannot be described
in  the language of normed spaces. We have to make use of
quasi-normed linear spaces in order to treat them
simultaneously.

A quasi-norm $\|\cdot\|_{\bf E}$ on linear space $\bf E$ is
a real-valued function on $\bf E$ such that for any
$f,f_{1}, f_{2} \in \EB$ the following holds true:  (1)$\|f\|_{\bf E}\geq 0;\>\>\>$
(2) $\|f\|_{\EB}=0  \Longleftrightarrow   f=0;\>\>\>$
(3)$\|-f\|_{\EB}=\|f\|_{\EB};\>\>\>$
(4) there exists some $C_{\EB} \geq 1$ such that
$\|f_{1}+f_{2}\|_{\EB}\leq C_{\EB}(\|f_{1}\|_{\EB}+\|f_{2}\|_{\EB}).\>\>$
Two quasi-normed linear spaces $\EB$ and $\FB$ form a
pair if they are linear subspaces of a common linear space
$\AB$ and the conditions
$\|f_{k}-g\|_{\EB}\rightarrow 0,$ and
$\|f_{k}-h\|_{\FB}\rightarrow 0$
imply equality $g=h$ (in $\AB$).
For any such pair $\EB,\FB$ one can construct the
space $\EB \cap \FB$ with quasi-norm
$
\|f\|_{\EB \cap \FB}=\max\left(\|f\|_{\EB},\|f\|_{\FB}\right)
$
and the sum of the spaces,  $\EB + \FB$ consisting of all sums $f_0+f_1$ with $f_0 \in \EB, f_1 \in \FB$, and endowed with the quasi-norm
$
\|f\|_{\EB + \FB}=\inf_{f=f_{0}+f_{1},f_{0}\in \EB, f_{1}\in
\FB}\left(\|f_{0}\|_{\EB}+\|f_{1}\|_{\FB}\right).
$

Quasi-normed spaces $\HB$ with
$\EB \cap \FB \subset \HB \subset \EB + \FB$
are called intermediate between $\EB$ and $\FB$.
If both $E$ and $F$ are complete the
inclusion mappings are automatically continuous. 
An additive homomorphism $T: \EB \rightarrow \FB$
is called bounded if
$
\|T\|=\sup_{f\in \EB,f\neq 0}\|Tf\|_{\FB}/\|f\|_{\EB}<\infty.
$
An intermediate quasi-normed linear space $\HB$
interpolates between $\EB$ and $\FB$ if every bounded homomorphism $T:
\EB+\FB \rightarrow \EB + \FB$
which is a bounded homomorphism of $\EB$ into
$\EB$ and a bounded homomorphism of $\FB$ into $\FB$
is also a bounded homomorphism of $\HB$ into $\HB$.
On $\EB+\FB$ one considers the so-called Peetre's $K$-functional
$$
K(f, t)=K(f, t,\EB, \FB)=\inf_{f=f_{0}+f_{1},f_{0}\in \EB,
f_{1}\in \FB}\left(\|f_{0}\|_{\EB}+t\|f_{1}\|_{\FB}\right).\label{K}
$$
The quasi-normed linear space $(\EB,\FB)^{K}_{\theta,q}$,
with parameters $0<\theta<1, \,
0<q\leq \infty$,  or $0\leq\theta\leq 1, \, q= \infty$,
is introduced as the set of elements $f$ in $\EB+\FB$ for which
\begin{equation}
\|f\|_{\theta,q}=\left(\int_{0}^{\infty}
\left(t^{-\theta}K(f,t)\right)^{q}\frac{dt}{t}\right)^{1/q} < \infty .\label{Knorm}
\end{equation}

It turns out that $(\EB,\FB)^{K}_{\theta,q}$
with the quasi-norm
(\ref{Knorm})  interpolates between $\EB$ and $\FB$.

\section{Approximation by Hardy-Steklov averages, K-functor and modulus of continuity}

We are going to prove  items (1) and (2) of Theorem \ref{Main}.
\begin{proof} First we prove the right-hand side of the inequality (\ref{main-ineq}).  The following simple lemma plays an important role in the roof.
\begin{lemma}
In any ring $\mathcal{R}$  with multiplicative identity $1$ the following formulas hold for $a_{1}, a_{2}, ....,a_{n}\in \mathcal{R},$
\begin{equation}\label{1}
a_{1}a_{2}...a_{n}-1=
a_{1}(a_{2}-1)+...+a_{1}a_{2}...a_{n-1}(a_{n}-1),
\end{equation}
\begin{equation}\label{2}
(a_{1}-1)a_{2}...a_{n}=(a_{1}-1)+a_{1}(a_{2}-1)+... +a_{1}a_{2}...   a_{n-1}(a_{n}-1).
\end{equation}
\end{lemma}
 By using the first formula we obtain

\begin{equation}\label{h-approx}
\left \|(-1)^{n(r+1)}H_{r}(s)f-f\right\|_{\mathbf{E}}=
\left \|\prod_{j=1}^{n}(-1)^{n+1}H_{j,r}(s)f-f\right\|_{\mathbf{E}}
\leq$$
 $$
c\sum_{j=1}^{n}\sup_{0\leq\tau_{j}\leq s/r}\left\|(T_{j}(\tau_{j})-I)^{r}f\right\|_{\mathbf{E}}\leq C\Omega^{r}(s,f) .
\end{equation}
To estimate $s^{r}\|H_{r}(s)f\|_{\mathbf{E}^{r}}$ we note that $\|H_{r}(s)f\|_{\mathbf{E}}\leq
C\|f\|_{\mathbf{E}}$.  According to Lemma  \ref{Mlemma}  the quantity
$\|D_{j_{1}}...D_{j_{r}}H_{r}(s)f\|_{\mathbf{E}}$ is estimated for $0\leq s\leq 1$ by
\begin{equation}\label{difestim}
s^{-l}\sup_{0\leq\tau _{j,k}\leq
s}\|\Delta_{j_{1},k_{1}}(s/r)...\Delta_{j_{l},k_{l}}(s/
r)T(\cdot)f\|_{\mathbf{E}},
\end{equation}
where  $T(\cdot)=
\prod_{j=1}^{n}\prod_{k=1}^{r}T_{j}(\tau_{j,k})$. By the
definition of
 $\Delta_{j,k}(s/r)$ the expression $\Delta_{j,k}(s/r)T(\cdot)$ differs from $T(\cdot)$
only in that in place of
  the factor $T_{j}(\tau_{j,k})$ the factor $T_{j}(s/r)-I$ appears.
  Multiple applications of the identity (\ref{2})    to the operator
$\Delta_{j_{1},k_{1}}(s/r )...\Delta_{j_{l}
  ,k_{l}}(s/r)T(\cdot)$ allow its expansion into a sum of operators each of
which is a product of not less than $l\leq r$ 
  of operators $T_{i}(\sigma_{i})-I, \>\>\>\sigma_{i} \in (0,s/r), \>\>\>1\leq i\leq n$. Consequently, (\ref{difestim}) is  dominated by a multiple of $s^{-l}\Omega^{l}(s,f)$.
By summing the estimates obtained above we arrive at the
inequality
\begin{equation}\label{up-estim}
K(s^{r},f, \mathbf{E}, \mathbf{E}^{r})\leq C\left( \sum_{l=1}^{n}s^{r-l}\Omega^{l}(s,f)+s^{r}\|f\|_{\mathbf{E}}\right),\>\>\>0\leq s\leq 1.
\end{equation}
Note, that by repeating the known proof for the classical modulus of continuity one can prove the inequality
$$
\Omega^{l}(s,f)\leq C\left(s^{l}\|f\|_{\mathbf{E}}+s^{l}\int_{s}^{1}\sigma^{-1-l}\Omega^{k+r}(\sigma, f)d\sigma\right),
$$
which implies 
$
s^{r-l}\Omega^{l}(s,f)\leq C\left( s^{r}\|f\|_{\mathbf{E}}+\Omega^{r}(s,f)\right).
$
By
applying  this inequality to (\ref{up-estim})  and taking into account the
inequality $K(s^{r}, f, \mathbf{E}, \mathbf{E}^{r})\leq \|f\|$, we obtain the right-hand side of the estimate (\ref{main-ineq}).
To prove the left-hand side of (\ref{main-ineq}) we first notice that  the following inequality holds
\begin{equation}\label{moduli-ineq}
\Omega^{r}(s, g)\leq Cs^{k}\sum_{1\leq j_{1}, ...j_{k}\leq n}\Omega^{r-k}\left(s, D_{j_{1}} ... D_{j_{k}}g\right),\>\>\>g\in \mathbf{E}^{k},\>\>\>C=C(k, r),\>\>\>k\leq r,
\end{equation}
which is an easy consequence of the identity (\ref{2}) and the identity $\left(T_{j}(t)-I\right)g=\int_{0}^{t}T_{j}(\tau)D_{j}gd\tau,\>\>g\in \mathcal{D}(D_{j})$. From here, for any $f\in \mathbf{E}, \>\>g\in \mathbf{E}^{r}$ we obtain 
$\Omega^{r}(s, f)\leq\Omega^{r}(s, f-g)+\Omega^{r}(s, g)\leq C\left(  \|f-g\|_{\mathbf{E}}+s^{r}\|g\|_{\mathbf{E}^{r}}\right)$. The first item of Theorem \ref{Main} is proven and it obviously implies second item of the same Theorem.

\end{proof}

\begin{remark}
The  proof shows that the left-hand side of (\ref{main-ineq}) holds true for any  finite set  of one-parameter strongly continuous bounded semigroups. 
\end{remark}

Below is the proof of items (3)-(5) of Theorem  \ref{Main}.

\begin{proof} We will need the following lemma.
\begin{lemma}
The following inequalities hold
\begin{equation}\label{I}
\|f\|_{\mathbf{E}^{k}}\leq C\|f\|_{\mathbf{E}}^{1-k/r}\|f\|_{\mathbf{E}^{r}}^{k/r},\>\>\>f\in \mathbf{E}^{r}, \>\>\>C=C(k,r),
\end{equation}
\begin{equation}\label{II}
K(s^{r},f, \mathbf{E}, \mathbf{E}^{r})\leq      Cs^{k}\|f\|_{\mathbf{E}^{k}},\>\>\>f\in \mathbf{E}^{k}, \>\>\>C=C(k,r).
\end{equation}

\end{lemma}
\begin{proof}
The first inequality follows from  its well-known one-dimensional version (see also \cite{Pes09a}). The second one follows from the right-hand estimate of (\ref{main-ineq}) and (\ref{moduli-ineq}).
\end{proof}
This lemma shows that one can use the Reiteration Theorem (see \cite{BB}, \cite{KPS}), which immediately implies item (3) of Theorem \ref{Main}. Next, let $\alpha>0,\>\>\>[\alpha]$ be a non-integer and its integer part respectively.  According to item (3) of Theorem \ref{Main} we have $\left( \mathbf{E}, \mathbf{E}^{r}\right)^{K}_{\alpha/r, q}=\left( \mathbf{E}^{[\alpha]}, \mathbf{E}^{r}\right)^{K}_{(\alpha-[\alpha])/(r-[\alpha]), q}$ and $\left( \mathbf{E}, \mathbf{E}^{1}\right)^{K}_{\alpha-[\alpha], q}=\left( \mathbf{E}, \mathbf{E}^{r-[\alpha]}\right)^{K}_{(\alpha-[\alpha])/(r-[\alpha]), q}$. Note, that $D_{j_{1}}D_{j_{2}}...D_{j_{[\alpha]}}$ is a continuous map from $\left( \mathbf{E}^{[\alpha]}, \mathbf{E}^{r}\right)^{K}_{(\alpha-[\alpha])/(r-[\alpha]), q}$ to $\left( \mathbf{E}, \mathbf{E}^{r-[\alpha]}\right)^{K}_{(\alpha-[\alpha])/(r-[\alpha]), q}$. All together it shows that if $f\in \left( \mathbf{E}, \mathbf{E}^{r}\right)^{K}_{\alpha/r, q}$ then $D_{j_{1}}D_{j_{2}}...D_{j_{[\alpha]}}f\in \left( \mathbf{E}, \mathbf{E}^{1}\right)^{K}_{\alpha-[\alpha], q}$ and 
\begin{equation}\label{(4)}
\left\|D_{j_{1}}D_{j_{2}}...D_{j_{[\alpha]}}f\right\|_{\left( \mathbf{E}, \mathbf{E}^{1}\right)^{K}_{\alpha-[\alpha], q}}\leq C\|f\|_{ \left( \mathbf{E}, \mathbf{E}^{r}\right)^{K}_{\alpha/r, q}}.
\end{equation}
Conversely,   let $D_{j_{1}}D_{j_{2}}...D_{j_{[\alpha]}}f\in \left( \mathbf{E}, \mathbf{E}^{1}\right)^{K}_{\alpha-[\alpha], q}=\left( \mathbf{E}^{[\alpha]}, \mathbf{E}^{r}\right)^{K}_{(\alpha-[\alpha])/(r-[\alpha]), q}.$ Then the right-hand estimate of (\ref{main-ineq}) and (\ref{moduli-ineq}) imply
\begin{equation}\label{(4')}
\|f\|_{ \left( \mathbf{E}, \mathbf{E}^{r}\right)^{K}_{\alpha/r, q}}\leq C\sum_{j_{1},...,j_{[\alpha]}=1}^{n}\left\|D_{j_{1}}D_{j_{2}}...D_{j_{[\alpha]}}f\right\|_{\left( \mathbf{E}, \mathbf{E}^{1}\right)^{K}_{\alpha-[\alpha], q}}.
\end{equation}
Inequalities (\ref{(4)}) and (\ref{(4')}) imply   item (4) of Theorem \ref{Main}. Proof of item (5) is similar. Theorem \ref{Main} is completely proved. 

\end{proof}

\section{Shannon sampling, Paley-Wiener  frames
and  abstract  Besov subspaces
  }\label{Hilbert}

\subsection{Paley-Wiener vectors in Hilbert spaces}

Consider a self-adjoint positive definite operator
$L$ in a Hilbert space $\mathbf{H}$.
Let $\sqrt{L}$ be the positive square root of $L$.
According to the spectral theory
for such operators \cite{BS}
there exists a direct integral of
Hilbert spaces $X=\int X(\lambda )dm (\lambda )$ and a unitary
operator $\mathcal{F}$ from $\mathbf{H}$ onto $X$, which
transforms the domains of $L^{k/2}, k\in \mathbb{N},$
onto the sets
$X_{k}=\{x \in X|\lambda ^{k}x\in X \}$
with the norm 
\begin{equation}\label{FT}
\|x(\lambda)\|_{X_{k}}= \left<x(\lambda),x(\lambda)\right>^{1/2}_{X(\lambda)}=\left (\int^{\infty}_{0}
 \lambda^{2k}\|x(\lambda )\|^{2}_{X(\lambda )} dm
 (\lambda ) \right )^{1/2}.
 \end{equation}
and satisfies the identity
$\mathcal{F}(L^{k/2} f)(\lambda)=
 \lambda ^{k} (\mathcal{F}f)(\lambda), $ if $f$ belongs to the domain of
 $L^{k/2}$.
 We call the operator $\mathcal{F}$ the Spectral Fourier Transform \cite{Pes6}, \cite{Pes00}. As known, $X$ is the set of all $m $-measurable
  functions $\lambda \mapsto x(\lambda )\in X(\lambda ) $,
  for which the following norm is finite:
$$\|x\|_{X}=
\left(\int ^{\infty }_{0}\|x(\lambda )\|^{2}_{X(\lambda )}dm
(\lambda ) \right)^{1/2} $$
For a function $F$ on $[0, \infty)$ which is bounded and measurable
with respect to $dm$
one can introduce the  operator $\FSL$
by using the formula
\begin{equation}\label{Op-function}
\FSL f=\mathcal{F}^{-1}F( \lambda)\mathcal{F}f,\>\>\>f\in \mathbf{H}.
\end{equation}
If $F$ is real-valued the operator $\FSL$ is self-adjoint.

\begin{remark} \label{rem:sqrtL} 
In many applications $L$ is a second-order differential operator and then $\sqrt{L}$ is a first-order  pseudo-differential operator.
\end{remark}
\begin{definition}\label{PWvector}
For $\sqrt{L}$  as above  we will  say that a vector $f \in\mathbf{H}$ belongs to the {\it Paley-Wiener space} $\PWoL$
if the support of the Spectral Fourier Transform
$\mathcal{F}f$ is contained in $[0, \omega]$.
\end{definition}
The next two facts are obvious.
\begin{theorem}The spaces $\PWoL$ have the following properties:
\begin{enumerate}
\item  the space $\PWoL$ is a linear closed subspace in
$\mathbf{H}$.
\item the space 
 $\bigcup _{ \omega >0}\PWoL$
 is dense in $\mathbf{H}$;
\end{enumerate}
\end{theorem}
Next we denote by $\mathbf{H}^{k}$  the domain  of $L^{k/2}$.
It is a Banach space,
equipped with the graph norm $\|f\|_{k}=\|f\|_{\mathbf{H}} +\|L^{k/2}f\|_{\mathbf{H}} $.
The next theorem contains generalizations of several results
from  classical harmonic analysis (in particular  the
Paley-Wiener theorem). It follows from our  results in
\cite{Pes00}.
\begin{theorem}\label{PWproprties}
The following statements hold:
\begin{enumerate}
\item (Bernstein inequality)   $f \in \PWoL$ if and only if
$ f \in \mathbf{H}^{\infty}=\bigcap_{k=1}^{\infty}\mathbf{H}^{k}$,
and the following Bernstein inequalities  holds true
\begin{equation}\label{Bern0}
\|L^{s/2}f\|_{\mathbf{H}} \leq \omega^{s}\|f\|_{\mathbf{H}}  \quad \mbox{for all} \, \,  s\in \mathbb{R}_{+};
\end{equation}
 \item  (Paley-Wiener theorem) $f \in \PWoL$
  if and only if for every $g\in\mathbf{H}$ the scalar-valued function of the real variable  $ t \mapsto
\langle e^{it\sqrt{L}}f,g \rangle $
 is bounded on the real line and has an extension to the complex
plane as an entire function of the exponential type $\omega$;
\item (Riesz-Boas interpolation formula) $f \in \PWoL$ if
and only if  $ f \in \mathbf{H}^{\infty}$ and the
following Riesz-Boas interpolation formula holds for all $\omega > 0$:
\begin{equation}  \label{Rieszn}
i\sqrt{L}f=\frac{\omega}{\pi^{2}}\sum_{k\in\mathbb{Z}}\frac{(-1)^{k-1}}{(k-1/2)^{2}}
e^{i\left(\frac{\pi}{\omega}(k-1/2)\right)\sqrt{L}}f.
\end{equation}
\end{enumerate}
\end{theorem}
\begin{proof}
(1) follows immediately from the definition and representation (\ref{FT}).  To prove (2) it is sufficient to apply the classical Bernstein inequality \cite{N}  in the  uniform norm on $\mathbb{R}$ to every  function $\langle e^{it\sqrt{L}}f,g\rangle,\>\>g\in\mathbf{H}$. To prove   (3) one has to apply the classical Riesz-Boas interpolation formula on $\mathbb{R}$, \cite{Pes15a}, \cite{N} to a function
$\langle e^{it\sqrt{L}}f,g\rangle$.
\end{proof}

\subsection{Frames in Hilbert spaces}

A family of vectors $\{\theta_{v}\}$  in a Hilbert space $\mathbf{H}$ is called a frame if there exist constants
$A, B>0$ such that
\begin{equation}
A\|f\|^{2}_{\mathbf{H}}\leq \sum_{v}\left|\left<f,\theta_{v}\right>\right|^{2}    \leq B\|f\|_{\mathbf{H}} ^{2} \quad \mbox{for all} \quad f\in \mathbf{H}.
\end{equation}
The largest $A$ and smallest $B$ are called lower and upper frame bounds.

The family of scalars $\{\left<f,\theta_{v}\right>\}$
represents a set of measurements of a vector $f$.
In order to resynthesize the vector $f$
from this collection  of measurements in a linear way
one has to find another
(dual) frame $\{\Theta_{v}\}$.
Then a reconstruction formula is
$
f=\sum_{v}\left<f,\theta_{v}\right>\Theta_{v}.
$
Dual frames are not unique in general.
Moreover it may be difficult to find a dual frame in concrete
situations.  
If $A=B=1$ the frame is said to be  tight   or Parseval.
Parseval frames are similar in many respects to orthonormal wavelet bases.  For example, if in addition all vectors $\theta_{v}$ are unit vectors, then the frame is an  orthonormal basis.
The main feature of Parseval frames is that
decomposing
and synthesizing a vector from known data are tasks carried out with
the same family of functions, i.e., the Parseval frame is its own dual frame.

\subsection{Sampling in abstract Paley-Wiener spaces}

We now assume that  the {\bf Assumption 3} is satisfied.  It meant that there exists a  $C>0$ and  $m_{0}\geq 0$ such that for any $0<\rho<1$ there exists a        set of functionals  $\mathcal{A}^{(\rho)}=\left\{\mathcal{A}_{ k}^{(\rho)}\right\},$  defined on $\mathbf{H}^{m_{0}}$, for which the inequalties (\ref{A}) hold true.

\begin{remark}
Following \cite{Pes01}, \cite{Pes04b} we call inequality (\ref{A}) a  {\it Poincar\'{e}-type inequality} since it is an estimate of the norm  of 
$f$ through the norm  of its ``derivative'' $L^{m/2}f$.
\end{remark}

 Let us introduce vectors  $\mu_{ k}\in\mathbf{H}$ such that
$
 \left<f,\mu_{ k}\right>=\mathcal{A}_{k}^{(\rho)}(f),\>\>f\in \mathbf{H}^{m},\>\>m> m_{0}.
$ Let  $\mathcal{P}_{\negthinspace \omega}$
 be the orthogonal projection of $\mathbf{H}$ onto
 $\PWoL$ and put
 \begin{equation}\label{functionals-phi}
 \phi^{\omega}_{ k}=\mathcal{P}_{\negthinspace \omega}\mu_{ k}.
\end{equation}
 Using  the Bernstein inequality (\ref{Bern0}) we obtain  the following statement.

 \begin{theorem}\label{PWSs}(Sampling Theorem)
Assume  that inequality  (\ref{A}) holds and for a given  $\omega>0$ and $\delta \in (0,1)$  pick a $\rho$ such that 
 \begin{equation}\label{rho}
\rho^{2m}=C^{-1}\omega^{-2m}\delta.
\end{equation}
Then  the family of vectors
$\{\phi^{\omega}_{k}\}$ is a frame for the Hilbert space $\PWoL$ and
\begin{equation}\label{frame-in-PW}
(1-\delta) \|f\|^{2}_{\mathbf{H}}\leq \sum_{k}\left|\left<f,\phi^{\omega}_{k}\right>\right|^{2}     \leq \|f\|^{2}_{\mathbf{H}},\>\>\>f\in \PWoL.
\end{equation}
The canonical dual frame $\{\Theta^{\omega}_{k}\}$  has the property   $\Theta^{\omega}_{k}\in  \PWoL$ and
provides the following reconstruction formulas
\begin{equation}
f=\sum_{k}\left<f,\phi^{\omega}_{k}\right>\Theta^{\omega}_{k}=\sum_{k}\left<f,\Theta^{\omega}_{k}\right>\phi^{\omega}_{k},\>\>\>f\in \PWoL.
\end{equation}
\end{theorem}

\subsection{Partitions of unity on the frequency side}

\label{subsect:part_unity_freq}

The construction of frequency-localized frames is  achieved via spectral calculus. The idea is to start from a partition of unity on the positive real axis. In the following, we will be considering two different types of such partitions, whose construction we now describe in some detail.

 Let $g\in C^{\infty}(\mathbb{R}_{+})$ be a non-increasing
 function such that $supp(g)\subset [0,\>  2], $ and $g(\lambda)=1$ for $\lambda\in [0,\>1], \>0\leq g(\lambda)\leq 1, \>\lambda>0.$
 We now let
$
 h(\lambda) = g(\lambda) - g(2 \lambda)~,
$ which entails $supp(h) \subset [2^{-1},2]$,
and use this to define
$
 F_0(\lambda) = \sqrt{g(\lambda)}~, F_j(\lambda) = \sqrt{h(2^{-j} \lambda)}~, j \ge 1~,
$
as well as $
 G_j(\lambda) = \left[F_j(\lambda)\right]^2=F_j^2(\lambda)~, j \ge 0~.$
As a result of the definitions, we get for all $\lambda \ge 0$ the equations
$
\sum_{j = 0}^n G_j(\lambda) = \sum_{j = 0}^n F_j^2(\lambda)
=  g(2^{-n} \lambda),
$
and as a consequence
$
\sum_{j \ge 0} G_j(\lambda) = \sum_{j \ge 0} F_j^2(\lambda) =  1~,\>\>\>\lambda\geq 0,
$
 with finitely many nonzero terms occurring in the sums for each
 fixed $\lambda$. 
We call the sequence $(G_j)_{j \ge 0}$ a {\bf (dyadic) partition of unity}, and $(F_j)_{j \ge 0}$ a {\bf quadratic (dyadic) partition of unity}.  As will become soon apparent, quadratic partitions are useful for the construction of frames.
 Using the spectral theorem one has
$
F_{j}^{2}\SLB  f=\mathcal{F}^{-1}\left(F_{j}^{2}(\lambda)\mathcal{F}f(\lambda)\right),\>\>\>j\geq 1,
$
and thus
\begin{equation} \label{eqn:quad_part_identity}
 f = \mathcal{F}^{-1}\mathcal{F}f(\lambda) =\mathcal{F}^{-1}\left(\sum_{j\geq 0}F_{j}^{2}(\lambda)\mathcal{F}f(\lambda)\right) = \sum_{j\geq 0} F_{j}^2\SLB f
\end{equation}  
Taking inner product with $f$ gives
$
\|F_{j}\SLB f\|^{2}_{\mathbf{H}}=\langle F_{j}^{2}\SLB f, f \rangle
$  
and
$
\|f\|_{\mathbf{H}}^2=\sum_{j\geq 0}\langle F_{j}^2 \SLB f,f\rangle=\sum_{j\geq 0}\|F_{j}\SLB f\|_{\mathbf{H}}^2 .
$
Similarly, we get the identity  $
 \sum_{j \geq 0} G_j \SLB f = f~.$
Moreover, since the functions $G_j,  F_{j}$, have their supports in  $
[2^{j-1},\>\>2^{j+1}]$, the elements $ F_{j} \SLB f $ and $G_j \SLB f$
 are bandlimited to  $[2^{j-1},\>\>2^{j+1}]$, whenever $j \ge 1$, and to $[0,2]$ for $j=0$.

 \subsection{Paley-Wiener frames in  Hilbert spaces}\label{Hilb}

Using the notation from above and Theorem \ref{PWSs}, one can describe  the following Paley-Wiener frame in an abstract Hilbert space  $\mathbf{H}$.

\begin{theorem}\label{frameH}(Paley-Wiener nearly Parseval frame in $\mathbf{H}$)

For a fixed $\delta\in (0,1)$ and $j\in \mathbb{N}$ let $\{\phi^{j}_{k}\}$ be a set of vectors described in Theorem \ref{PWSs} that correspond to $\omega=2^{j+1}$. 
Then for functions $F_{j}$    the family of Paley-Wiener  vectors
$
\Phi^{j}_{k}= F_{j} \SLB \phi^{j}_{k}
$
has the following properties:
\begin{enumerate}
\item Each vector $\Phi^{j}_{k}$ belongs  to  $\bPW_{[2^{j-1},\>2^{j+1}]}(\sqrt{L}) ,\>\> j \in   N, \>k=1,... .$
\item   The family $\left\{\Phi^{j}_{k}\right\}$ is  a frame in  $\mathbf{H}$ with constants $1-\delta$ and $1$:
\begin{equation}
(1-\delta)\|f\|_{\mathbf{H}}^2\leq \sum_{j\geq 0}\sum_{k}\left|\left< f, \Phi^{j}_{k}\right>\right|^{2}\leq \|f\|_{\mathbf{H}}^2,\>\>\>f \in\mathbf{H}.
\end{equation}
\item  The canonical dual frame $\{\Psi^{j}_{k}\}$
also consists of bandlimited  vectors $\Psi^{j}_{k}\in \bPW_{[2^{j-1},\>2^{j+1}]}\SLB ,\>\>j\in  [0,\>\infty), \>k=1,...,$ and has frame bounds $A=1,\>\>B=(1-\delta)^{-1}.$ 

\item The reconstruction formulas hold for every $f\in \mathbf{H}$

$
f=\sum_{j}\sum_{k}\left<f,\Phi^{j}_{k}\right>\Psi^{j}_{k}=\sum_{j}\sum_{k}\left<f,\Psi^{j}_{k}\right>\Phi^{j}_{k}.
$

\end{enumerate}
\end{theorem}

The last two items here follow from the first two and general properties  of frames.
We also note that  for reconstruction of a Paley-Wiener vector  from  a set of samples one can use, besides dual frames, the  variational (polyharmonic) splines in Hilbert spaces developed in   \cite{Pes01}.

\section{Besov subspaces in Hilbert spaces}\label{Bes}

\subsection{Approximation spaces}\label{Appr}

Let us introduce another functional on $\EB+\FB$,
where $\EB$ and $\FB$ form a pair of quasi-normed linear spaces
$
\mathcal{E}(f, t)=\mathcal{E}(f, t, \mathbf{E},  \mathbf{F})=\inf_{g\in \FB,
\|g\|_{\FB}\leq t}\|f-g\|_{\EB}.
$
\begin{definition}
The approximation space $\mathcal{E}_{\alpha,q}(\EB, \FB),
0<\alpha<\infty, 0<q\leq \infty $ is the quasi-normed linear spaces
of all $f\in \EB+\FB$ for which the quasi-norm
\begin{equation}
\| f \|_{\mathcal{E}_{\alpha,q}(\EB, \FB)} = \left(\int_{0}^{\infty}\left(t^{\alpha}\mathcal{E}(f,
t)\right)^{q}\frac{dt}{t}\right)^{1/q} 
\end{equation}
is finite. 
\end{definition}

The next theorem   represents a very abstract version of what is  known as an Equivalence Approximation Theorem \cite{PS}, \cite{BS}. In the form it is stated below it was proved in \cite{KP}.

\begin{theorem}\label{equivalence-interpolation}
 Suppose that $\mathcal{T}\subset \FB \subset \EB$ are quasi-normed
linear spaces and $\EB$ and $\FB$ are complete.
If there exist $C>0$ and $\beta >0$ such that
the following Jackson-type inequality is satisfied 
$
t^{\beta}\mathcal{E}(t,f,\mathcal{T},\EB)\leq C\|f\|_{\FB},
\>\>t>0, \>\> f \in \FB,$
 then the following embedding holds true
\begin{equation}\label{imbd-1}
(\EB,\FB)^{K}_{\theta,q}\subset
\mathcal{E}_{\theta\beta,q}(\EB, \mathcal{T}), \quad \>0<\theta<1, \>0<q\leq \infty.
\end{equation}
If there exist $C>0$ and $\beta>0$
such that 
the following Bernstein-type inequality holds
$
\|f\|_{\FB}\leq C\|f\|^{\beta}_{\mathcal{T}}\|f\|_{\EB}
,\>\> f\in \mathcal{T},
$
then the following embedding holds true
\begin{equation}\label{imbd-2}
\mathcal{E}_{\theta\beta, q}(\EB, \mathcal{T})\subset
(\EB, \FB)^{K}_{\theta, q}  , \quad 0<\theta<1, \>0<q\leq \infty.
\end{equation}
\end{theorem}\label{intthm}

\subsection{Besov subspaces in  Hilbert spaces}\label{AbstractBesov}

According to (\ref{interp}) we introduce $
\mathcal{B}_{\mathbf{H},q}^\alpha \SLB =( \mathbf{H},  \mathbf{H}^{r})^{K}_{\theta, q},\>\>\> 0<\theta=\alpha/r<1,\>\>\>
1\leq q\leq \infty.
$
We also introduce a notion of best approximation: 
$$
\mathcal{E}(f,\omega)=\inf_{g\in
\PWoL}\|f-g\|_{\mathbf{H}}.\label{BA1}
$$
Our goal is to apply Theorem \ref{equivalence-interpolation} in the situation where $\mathbf{E}=\mathbf{H}$,  $\>\mathbf{F}=\mathbf{H}^{r}$ and $\mathcal{T}=\PWoL $ is a natural abelian group as the additive group of a vector space, with the quasi-norm
$
 \| f \|_{\mathcal{T}} = \inf \left \{ \omega'>0~: f \in \mathbf{PW}_{\mathbf{\omega}'}\left(\sqrt{L}\right) \right\}~.
$
To be more precise it is the space
of finite sequences of Fourier coefficients $\mathbf{c}=(c_{1},...c_{m})\in \PWoL$
 where $m$ is the greatest index such that the eigenvalue $\lambda_{m}\leq \omega$.
 For a $\mathbf{c}=(c_{1},...c_{m})\in \PWoL$ the quasi-norm is defined as  $\>\>\>
 \|\mathbf{c}\|_{E_{\omega}(L)} =\max\left\{\sqrt{\lambda_{j}}: c_{j}\neq 0, \>\>c_{j+1}=...=c_{m}=0\right\}.
 $
\begin{remark}
Let us emphasize  that the reason we need the language of  quasi-normed spaces is because $\| \cdot \|_{\mathcal{T}} $ is clearly not a norm, only a quasi-norm on $ \PWoL$. 
\end{remark}
The Plancherel Theorem allows us to verify a  generalization of the Bernstein inequality for bandlimited functions in $ f \in \PWoL$. One can prove the following statement (see \cite{Pes3}, \cite{Pes00}).

\begin{lemma} 
A vector $f$ belongs to the space $ \PWoL$ if and only if the following Bernstein inequality holds
$
\|L^{r/2}f\|_{\mathbf{H}}\leq \omega ^{r}\|f\|_{\mathbf{H}},\>\>\>r\in \mathbb{R}_{+}.
$
\end{lemma}
One also has an analogue of the Jackson inequality (see  \cite{Pes3}, \cite{Pes00})
$
\mathcal{E}(f,\omega)\leq \omega^{-r}\|f\|_{\mathbf{H}^{r}},\>\>\>f\in \mathbf{H}^{r}.
$
These two inequalities and Theorem \ref{equivalence-interpolation}  imply the following result (compare to  \cite{Pes8}, \cite{Pes11}).

\begin{theorem} \label{approx}
For $\alpha>0, 1\leq q\leq\infty$ the norm of
 $\mathcal{B}_{\mathbf{H},q}^{\alpha}\SLB$,
  is equivalent to
\begin{equation}
\|f\|_{\mathbf{H}}+\left(\sum_{j=0}^{\infty}\left(2^{j\alpha }\mathcal{E}(f,
2^{j})\right)^{q}\right)^{1/q}.
\end{equation}
\label{maintheorem1}
\end{theorem}

Let the functions $F_{j}$ be as in Subsection \ref{subsect:part_unity_freq}. 

\begin{theorem}\label{projections}
For $\alpha>0, 1\leq q\leq\infty$ the norm of
 $\mathcal{B}_{\mathbf{H},q}^{\alpha}\SLB$,
  is equivalent to

\begin{equation}
f \mapsto \left(\sum_{j=0}^{\infty}\left(2^{j\alpha
}\left \|F_j\SLB f\right \|_{\mathbf{H}}\right)^{q}\right)^{1/q},
\label{normequiv-1}
\end{equation}
  with the standard modifications for $q=\infty$.
\end{theorem}

\begin{proof}

We obviously have
$
\mathcal{E}(f, 2^{l})\leq \sum_{j> l} \left \|F_j \SLB f\right \|_{\mathbf{H}}.
$
By using a discrete version of Hardy's inequality \cite{BB} we obtain the estimate
\begin{equation} \label{direct}
\|f\|+\left(\sum_{l=0}^{\infty}\left(2^{l\alpha }\mathcal{E}(f,
2^{l})\right)^{q}\right)^{1/q}\leq C \left(\sum_{j=0}^{\infty}\left(2^{j\alpha
}\left \|F_j \SLB f\right \|_{\mathbf{H}}\right)^{q}\right)^{1/q}.
\end{equation}
Conversely,
 for any $g\in \bPW_{2^{j-1}} \SLB$ we have
$
\left\|F_j\SLB f\right\|_{\mathbf{H}}=\left\|F_{j}\SLB (f-g)\right\|_{\mathbf{H}}\leq \|f-g\|_{\mathbf{H}}.
$
This implies the estimate
$
\left\|F_j\SLB f\right\|_{\mathbf{H}}\leq \mathcal{E}(f,\>2^{j-1}),
$
which shows that the inequality opposite to (\ref{direct}) holds.
 The proof is complete.
\end{proof}

 \begin{theorem}\label{framecoef}
For $\alpha>0, 1\leq q\leq\infty$ the norm of
 $\mathcal{B}_{\mathbf{H},q}^{\alpha}\SLB$
  is equivalent to
\begin{equation}
 \left(\sum_{j=0}^{\infty}2^{j\alpha q }
\left(\sum_{k}\left|\left<f,\Phi^{j}_{k}\right>\right|^{2}\right)^{q/2}\right)^{1/q}\asymp \|f\|_{B_{q}^{\alpha}},
\label{normequiv}
\end{equation}
  with the standard modifications for $q=\infty$.
\end{theorem}
\begin{proof}
For   $f\in \mathbf{H}$  and operator $F_{j}\SLB$ 
we apply (\ref{frame-in-PW}) to $F_{j}\SLB f\in \bPW_{2^{j+1}}\SLB$ to obtain
\begin{equation}
(1-\delta)\left \|F_j\SLB f\right \|_{\mathbf{H}}^{2}\leq
\sum_k\left|\left< F_{j}\SLB f, \phi^{j}_{k}\right>\right|^{2}\leq
\left \|F_j\SLB f\right \|_{\mathbf{H}}^{2}.
\end{equation}
 Since $\Phi^{j}_{k}=F_{j}\SLB \phi^{j}_{k}$
 we obtain the following inequality
$$
\sum_{k}\left|\left<f,\Phi^{j}_{k}\right>\right|^{2}\leq \left \|F_j\SLB f\right \|_{\mathbf{H}}^{2}\leq \frac{1}{1-\delta}
\sum_{k}\left|\left<f,\Phi^{j}_{k}\right>\right|^{2}
\quad \mbox{for all} \, \, \, f\in \mathbf{H}.
$$
Our statement follows now from Theorem \ref{projections}.
\end{proof}

============================

==============================

\section{Applications}

\subsection{ Analysis on  $\mathbb{S}^{d}$}

We will specify the general setup in the case of standard unit sphere. 
Let 
$
\mathbb{S}^{d}=\left\{x\in \mathbb{R}^{d+1}: \|x\|=1\right\}.
$
Let $\mathcal{P}_{n}$ denote the space of spherical harmonics of degree $n$, which are restrictions to $\mathbb{S}^{d}$ of harmonic homogeneous polynomials of degree $n$ in $\mathbb{R}^{d}$. The Laplace-Beltrami operator $\Delta_{\mathbb{S}^{d}}$ on $\mathbb{S}^{d}$  is a restriction of the regular Laplace operator $\Delta$ in $\mathbb{R}^{d}$. Namely,
 $
\Delta_{\mathbb{S}^{d}}f(x)=\Delta \widetilde{f}(x),\>\>x\in \mathbb{S}^{d},
$
where $\widetilde{f}(x)$ is the homogeneous extension of $f$: $\>\>\widetilde{f}(x)=f\left(x/\|x\|\right)$. Another way to compute $\Delta_{\mathbb{S}^{d}}f(x)$ is to express both $\Delta_{\mathbb{S}^{d}}$ and $f$ in a spherical coordinate system.
Each $\mathcal{P}_{n}$ is the eigenspace of  $\Delta_{\mathbb{S}^{d}}$ that corresponds to the eigenvalue $-n(n+d-1)$. Let $Y_{n,l},\>\>l=1,...,l_{n}$ be an orthonormal basis in $\mathcal{P}_{n}$.

Let $e_{1},...,e_{d+1}$ be the standard orthonormal basis in $\mathbb{R}^{d+1}$.  If $SO(d+1)$ and $SO(d)$ are the groups of rotations of $\mathbb{R}^{d+1}$ and  $\mathbb{R}^{d}$ respectively then $\mathbb{S}^{d}=SO(d+1)/SO(d)$. 
On $\mathbb{S}^{d}$ we consider vector fields
$
X_{i,j}=x_{j}\partial_{x_{i}}-x_{i}\partial_{x_{j}}
$
which are generators of one-parameter groups of rotations   $\exp tX_{i,j}\in SO(d+1)$ in the plane $(x_{i}, x_{j})$. These groups are defined by the formulas for $\tau\in \mathbb{R}$, 
$$
\exp \tau X_{i,j}\cdot (x_{1},...,x_{d+1})=(x_{1},...,x_{i}\cos \tau -x_{j}\sin \tau ,..., x_{i}\sin \tau +x_{j}\cos \tau ,..., x_{d+1})
$$
Let $T_{i,j}(\tau)$ be a one-parameter group which is a representation of $\exp \tau X_{i,j}$ in the space $L_{p}(\mathbb{S}^{d})$. It acts on $f\in L_{p}(\mathbb{S}^{d})$ by the following formula
$$
T_{i,j}(\tau)f(x_{1},...,x_{d+1})=f(x_{1},...,x_{i}\cos \tau -x_{j}\sin \tau ,..., x_{i}\sin \tau +x_{j}\cos \tau ,..., x_{d+1}).
$$
Let $D_{i,j}$ be a generator of $T_{i,j}$ in $L_{p}(\mathbb{S}^{d})$.  The Laplace-Beltrami operator $\Delta_{\mathbb{S}^{d}}$   can be identified with the operator $L=  \sum_{i<j}D_{i,j}^{2}.$   One can easily illustrate our results by describing  norms  in Sobolev and Besov spaces on $\mathbb{S}^{d}$ in terms of  operators $T_{i,j}$ and $D_{i,j}$. In this situation role of Paley-Wiener subspaces is played by subspaces $\mathcal{P}_{n}$. A  set of functionals $\left\{  \mathcal{A}_{k}^{(\rho)}  \right\}_{k\in \mathcal{K}_{\rho}}$ described in (\ref{A}), (\ref{rho}) can be represented by a set of Dirac measures at  nodes $\{x_{k}^{(\rho)}\}$ "nearly uniformly" distributed over the sphere $\mathbb{S}^{d}$. 
\subsection{Compact homogeneous manifolds}
It should be noted,  that a similar situation holds on any compact homogeneous manifolds  $M=G/K$ where $G$ is a compact Lie group and $K$ is its closed subgroup. Moreover,  in this case description of Besov spaces in terms of  approximation by Paley-Wiener  vectors (eigenfunctions of a corresponding Laplace-Beltrami operator) and in terms of frame coefficients can be extended to  any $1\leq p\leq \infty$ \cite{gp}, \cite{Pes14a}, \cite{FFP}.  
A  set of functionals $\left\{  \mathcal{A}_{k}^{(\rho)}  \right\}_{k\in \mathcal{K}_{\rho}}$ can be represented by a set of Dirac measures  (or some other functionals \cite{Pes04b}) at  a set of nodes $\{x_{k}^{(\rho)}\}$  "nearly uniformly" distributed over the manifold $M$ with the spacing comparable to $\rho>0$. The Weyl's asymptotic formula \cite{Hor} implies  \cite{Pes04a}  that a  rate of sampling which is given by (\ref{rho}) is essentially optimal.

\subsection{Non-compact symmetric spaces}

Our framework also holds on non-compact symmetric spaces  \cite{Pes08}-\cite{Pes13b}. Besov spaces  can be characterized either using corresponding modulus of continuity \cite{Pes09c} or by approximation by Paley-Wiener vectors \cite{Pes09a} or in terms of frames \cite{Pes13b}.  In this situation Paley-Wiener functions which admit explicit description  in terms of the Helgason-Fourier transform \cite{Pes09a}  can be represented by a set of Dirac measures (or even more general functionals \cite{Pes04b}) at  a set of nodes $\{x_{k}^{(\rho)}\}$   "nearly uniformly" distributed over the manifold $M$ with the spacing comparable to $\rho>0$.

\begin{remark}
It should be noted that in the case of non-compact symmetric spaces our approach leads to Sobolev and Besov spaces which are different from the conventional ones generated by the Laplace-Beltrami operator associated with the natural metric. 
\end{remark}
\bigskip

\bibliographystyle{amsalpha}

\end{document}